\documentclass{amsart}[12pt]

\renewcommand{\S}{\Sigma}

\usepackage{enumerate}
\usepackage[hidelinks]{hyperref}

\usepackage{soul}
\usepackage{yfonts} %
\usepackage{amssymb} %
\usepackage{amsthm}
\usepackage{array}
\usepackage{booktabs}%
\usepackage{hhline}%
\usepackage{xy} %
\usepackage{epsfig}%
\usepackage{color}%
\usepackage{upgreek}
\usepackage[english]{babel}
\usepackage{epigraph}%
\usepackage{fancybox}%
\usepackage{shadow}
\usepackage{afterpage}
\usepackage{mathrsfs}
\usepackage{enumitem}
\usepackage{tabularx}
\usepackage{subcaption}
\usepackage{graphicx}
\usepackage{type1cm}
\usepackage{eso-pic}
\usepackage{color}
\usepackage{upgreek}
\usepackage[foot]{amsaddr}
\usepackage{verbatim}

\newtheorem{theorem}{Theorem}
\newtheorem{proposition}[theorem]{Proposition}

\newtheorem{lemma}[theorem]{Lemma}
\newtheorem{claim}[theorem]{Claim}
\theoremstyle{definition}
\newtheorem{definition}[theorem]{Definition}

\theoremstyle{plain}

\newcommand{\vt}{\vspace{.1cm}}

\newcommand{\R}{\mathbb{R} }

\newcommand{\N}{\mathbb{N} }
\newcommand{\h}{\mathbb{H} }

\renewcommand{\rho}{\varrho}

\renewcommand{\Theta}{\varTheta}
\renewcommand{\Lambda}{\varLambda}
\renewcommand{\Sigma}{\varSigma}
\renewcommand{\tau}{\uptau}
\usepackage{amsmath}

\newcommand{\overbar}[1]{\mkern 1.5mu\overline{\mkern-1.5mu#1\mkern-1.5mu}\mkern 1.5mu}

\makeatletter
\newcommand{\tpitchfork}{%
 \vbox{
 \baselineskip\z@skip
 \lineskip-.52ex
 \lineskiplimit\maxdimen
 \m@th
 \ialign{##\crcr\hidewidth\smash{$-$}\hidewidth\crcr$\pitchfork$\crcr}
 }%
}
\makeatother

\begin{document}

\title[Rotators-Translators to Mean Curvature Flow]{Rotators-Translators to \\ Mean Curvature Flow  in  $\h^2\times\R$}
\author{R. F. de Lima \and A. K. Ramos \and J. P. dos Santos}
\address[A1]{Departamento de Matem\'atica - UFRN}
\email{ronaldo.freire@ufrn.br}
\address[A2]{Departamento de Matemática Pura e Aplicada - UFRGS}
\email{alvaro.ramos@ufrgs.br}
\address[A3]{Departamento de Matem\'atica - UnB}
\email{joaopsantos@unb.br}
\thanks{The second and third authors were partially supported by the National Council for
Scientific and Technological Development – CNPq}
\maketitle

\begin{abstract}
  We establish the existence of one-parameter families of helicoidal surfaces
  of $\mathbb H^2\times\mathbb R$ which, under mean curvature flow,
  simultaneously rotate about a vertical axis and translate vertically.

  \vspace{.15cm}

\vspace{.1cm}

\end{abstract}

\section{Introduction}
Given an orientable Riemannian $3$-manifold $\overbar M$, let
$\mathcal G:=\{\Gamma_t\,;\, t\in {\R}\}$ 
be a  one-parameter subgroup of its group of isometries, 
and denote by $\xi$ the Killing field determined by $\mathcal G$.
In this setting, a surface
$\Sigma$ of $\overbar M$ with unit normal $\eta$ and mean curvature $H$
is called a $\mathcal G$-\emph{soliton} to mean curvature flow
if the equality
\begin{equation} \label{eq-main}
H=\langle\xi,\eta\rangle
\end{equation}
holds everywhere on $\Sigma$. It is well known that,
under mean curvature flow (MCF), a $\mathcal G$-soliton $\Sigma$
moves by the actions of the isometries of $\mathcal G$ (see, e.g.,~\cite{hungerbuhler-smoczyk}).
More precisely, if $X_t\colon M\to\overbar M$, $t\in[0,T)$, is
the MCF in $\overbar M$ whose initial condition $\S:=X_0(M)$
is a $\mathcal G$-soliton, then $X_t(M)=\Gamma_t(\S)$ for all $t\in[0,T)$.

The most considered $\mathcal G$-solitons in Euclidean space $\R^3$, called
\emph{translators}, are those whose associated group $\mathcal G$ of isometries
consists of translations in a fixed direction.
There are also the \emph{rotators}, which are those whose
associated group $\mathcal G$ consists of rotations of $\R^3$ about a fixed axis.
In~\cite{halldorsson}, Halldorsson obtained one-parameter families of helicoidal rotators in $\R^3$,  which are also
translators.

In this note, inspired by
Halldorsson's work, {for each $h>0$}, we establish the existence of {a} one-parameter family of
helicoidal rotators-translators to MCF in $\h^2\times\R$ {of pitch $h$}
(see Definition~\ref{def-helicoidal} in Section~\ref{sec-proof}), where
$\h^2$ is the hyperbolic plane.
The results read as follows.

\begin{theorem} \label{th-main}
For any $h>0,$ there exists a one-parameter family of complete 
rotators to {\rm MCF} in $\h^2\times\R$ which are helicoidal surfaces of pitch $h.$ For each such surface,
the trace of the generating curve in $\h^2$ consists of two unbounded properly embedded arms
centered at a point $o\in\h^2$, with
each arm spiraling around  $o$ (Fig.~{\rm\ref{fig-main}}).
\end{theorem}

\begin{theorem} \label{th-translator}
Let $\Sigma=X(\R^2)$ be a helicoidal surface of pitch $h>0$ in $\h^2\times\R.$
Then, $\Sigma$ is a rotator to {\rm MCF}
if and only if $\S$ is a translator to {\rm MCF} with respect to the Killing field $\xi=-h\partial_t$, where
$\partial_t$ is the gradient of the height function of $\h^2\times\R$.
\end{theorem}

\begin{figure}[h]
\centering
\includegraphics[width=0.48\textwidth]{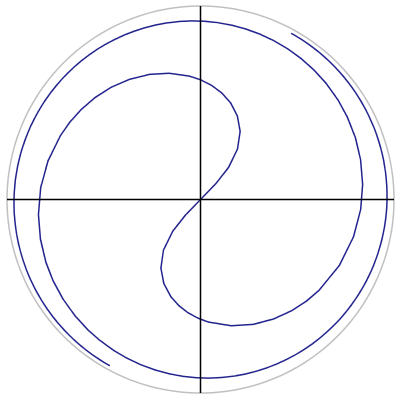}
\includegraphics[width=0.48\textwidth]{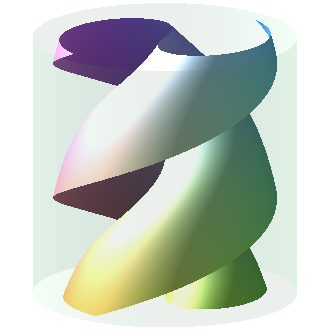}
\caption{\small A $1$-pitched helicoidal rotator-translator in $\h^2\times\R$ (right) and its
generating curve in the Poincaré {disk} $\h^2$ (left).}
\label{fig-main}
\end{figure}

It should be mentioned that helicoidal rotators-translators to MCF have been considered in other ambient $3$-spaces,
such as the Heisenberg space ${\rm Nil}_3$ (cf.~\cite{pipoli}), and the hyperbolic
space $\h^3$ (cf.~\cite{dLRS}).

\section{Proof of Theorems~\ref{th-main} and~\ref{th-translator}} \label{sec-proof}

Let $\mathbb L^3=(\R^3,\langle \,, \rangle)$ be the Lorentzian space,
where
\[
\langle \,,\, \rangle=ds^2:=dx_1^2+ds_2^2-dx_3^2,
\]
and consider the hyperbolic plane $\h^2\subset\mathbb L^3$ as
\[
\h^2=\{p\in \mathbb L^3\,;\, \langle p,p\rangle=-1, \,\, \langle p,e_3\rangle <0\}
\]
with the induced Lorentzian metric, where $e_3=(0,0,1)\in\mathbb L^3.$
In this model, the tangent plane of $\h^2$ at $p\in\h^2$ is the subspace
$\{p\}^\perp$ of $\mathbb L^3,$ that is,
\[
T_p\h^2=\{w\in\mathbb L^3\,;\, \langle w,p\rangle=0\}.
\]

Consider the product $\h^2\times\R$ endowed with its standard metric
\[
\langle \,,\,\rangle_{\h^2}+dt^2,
\]
together with its standard embedding in $\mathbb L^3\times\R.$
The rotators to MCF we shall consider are those defined by
the one-parameter group  $\mathcal G=\{\Gamma_t\,;\,t\in\R\}$
of rotations of $\h^2\times\R$ about
the axis $\ell:={\{ (0,0,1) \}}\times\R\subset\h^2\times\R,$ that is,
\[
\Gamma_t=\left[
 \begin{array}{ccc}
 e^{tJ} & & \\
 & 1 & \\
 & & 1
 \end{array}
\right],
\]
where $J(x_1,x_2)=(-x_2,x_1)$, $(x_1,x_2)\in\R^2$.
The Killing field associated to $\mathcal G$ is
\[
\xi(p):=\frac{\partial\Gamma_t}{\partial t}|_{t=0}(p)=J\pi(p), \,\,\, p\in\h^2\times\R,
\]
where $\pi(p)$ denotes the projection of $p$ over $\R^2\times{\{(0,0,0)\}}\subset\mathbb L^3.$

Therefore, considering equality~\eqref{eq-main}, we have that an oriented surface $\Sigma$ of $\h^2\times\R$
with unit normal $\eta$ is a rotator to MCF (with rotation axis $\ell$) if and
only if its mean curvature function $H$ satisfies
\begin{equation} \label{eq-rotatorH2xR}
H(p)=\langle J\pi(p),\eta(p)\rangle \,\,\,\forall p\in\Sigma.
\end{equation}

{In order to}
introduce the helicoidal surfaces of
$\h^2\times\R$ with axis $\ell$,
consider a differentiable curve $\sigma:I\subset\R\rightarrow\h^2$,
{written as}
\begin{equation}\label{eqsigmacoord}
\sigma=(\alpha,\phi),
\end{equation}
where $\alpha\colon I\rightarrow\R^2$ is a regular curve parameterized by
{arc length}, and $\phi$ is a differentiable function on the open interval $I.$ Notice that we are identifying
the plane $\{e_3\}^\perp\subset\mathbb L^3$ with the Euclidean plane $\R^2$,
so that the local theory of plane curves in $\R^2$ applies to $\alpha.$
Keeping that in mind, set
\[
T=\alpha' \quad\text{and}\quad N=JT,
\]
and recall that the curvature of $\alpha$ is the function
$k:=\langle\alpha'',N\rangle,$
which satisfies the \emph{Frenet equations} $T'=kN$ {and} $N'=-kT$.
We use $T$ and $N$ to define the following functions,
which will play a fundamental role in the sequel.
\begin{equation} \label{eq-tau&mu}
\tau:=\langle\alpha,T\rangle \quad\text{and}\quad \mu:=\langle\alpha,N\rangle.
\end{equation}
From the Frenet equations, the derivatives of $\tau$ and $\mu$ satisfy:
\begin{equation} \label{eq-tau'&mu'}
\tau'=1+k\mu \quad\text{and}\quad \mu'=-k\tau.
\end{equation}

\begin{definition} \label{def-helicoidal}
Given  $h>0$ and a differentiable curve
$\sigma=(\alpha,\phi)$ in $\h^2,$ we call a parameterized surface
$\Sigma=X(\R^2)\subset\h^2\times\R$ a \emph{helicoidal} surface generated by $\sigma$ {(as in~\eqref{eqsigmacoord})} with \emph{pitch} $h,$
if the parameterization $X:\R^2\rightarrow\h^2\times\R$ writes as
\begin{equation} \label{eq-parametrization}
X(u,v)=(e^{vJ}\alpha(u),\phi(u),hv), \,\,\,(u,v)\in\R^2.
\end{equation}
\end{definition}

We proceed now by determining  the mean curvature function of
a helicoidal surface in terms of its pitch $h$ and the functions
$\tau$ and $\mu$ defined in~\eqref{eq-tau&mu}. For that, we will use
the {well-known} formula that {expresses} $H$ with respect to the coefficients of
the first and second fundamental forms; namely
\[
H=\frac{Eg-2fF+Ge}{2(EG-F^2)}\cdot
\]

For a parameterization $X$ as in~\eqref{eq-parametrization}, we have
\begin{equation} \label{eq-xuxv}
\frac{\partial X}{\partial u}(u,v)=(e^{vJ}T(u),\phi'(u),0)\quad\text{and}\quad
\frac{\partial X}{\partial v}(u,v)=(e^{vJ}J\alpha(u),0,h).
\end{equation}
Besides,  since $\sigma$ satisfies $\langle \sigma,\sigma\rangle=-1$, we
have that $\langle \sigma,\sigma'\rangle=0$. Thus,
setting $r^2:=\tau^2+\mu^2$, the following equalities hold true.
\begin{equation}\label{eq-phiandphi'}
\phi^2=1+r^2 \quad\text{and}\quad \phi\phi'=\tau.
\end{equation}

Therefore, the coefficients of the first fundamental form of $X$ are
\[
E=\frac{1+\mu^2}{\phi^2}, \quad F=-\mu, \quad \text{and} \quad G=r^2+h^2.
\]
Also, it is easily seen that a unit normal to $\Sigma=X(\R^2)$ is
\begin{equation} \label{eq-normal}
\eta:=\rho(e^{vJ}(aT+bN),\mu, c), \quad \rho=(a^2+b^2-\mu^2+c^2)^{-1/2},
\end{equation}
where $a, b,$ and $c$ are the following functions of $u,$
\begin{equation} \label{eq-abc}
a := \mu\phi', \quad
b := \frac{1+\mu^2}{\phi},\quad\text{and}\quad
c := -\frac{\phi'}{h}\cdot
\end{equation}

Regarding the second derivatives of $X,$ we have
\begin{equation}\label{eq-xuu}
 \begin{aligned}
 X_{uu}(u,v) &= (e^{vJ}k(u)N(u),\phi''(u), 0),\\[1ex]
 X_{uv}(u,v) &= (e^{vJ}N(u),0, 0),\\[1ex]
 X_{vv}(u,v) &= (-e^{vJ}\alpha(u),0, 0),
 \end{aligned}
\end{equation}
where $k$ is the curvature function of $\alpha.$
Therefore, the coefficients of the second fundamental form of $X$ are
\[
e=\rho(bk-\phi''\mu), \quad f=\rho b, \quad\text{and}\quad
g=-\rho(a\tau+b\mu),
\]
which implies that the mean curvature $H$ of $\Sigma$ at $X(u,v)$ is
\begin{equation} \label{eq-H}
H=\rho\frac{\phi^2(bk-\phi''\mu)(h^2+r^2)+2b\mu\phi^2-(a\tau+b\mu)(1+\mu^2)}{{2}(\tau^2+h^2(1+\mu^2))}\cdot
\end{equation}

By differentiating the second equality in \eqref{eq-phiandphi'}, one gets
\[
\phi''=\frac{1}{\phi}\left(1+k\mu-\frac{\tau^2}{\phi^2}\right)=\frac{1}{\phi}\left(k\mu+\frac{1+\mu^2}{\phi^2}\right).
\]
Hence, from \eqref{eq-abc}, we have
\begin{eqnarray} \label{eq-bk-phi''mu}
 bk-\phi''\mu &=& \frac{(1+\mu^2)}{\phi}k-\frac{\mu}{\phi}\left(k\mu+\frac{1+\mu^2}{\phi^2}\right) \nonumber\\
 &=& \frac{k}{\phi}-\frac{\mu(1+\mu^2)}{\phi^3}\cdot
\end{eqnarray}
In addition, a direct computation gives 
\begin{equation}\label{eq-expression2}
 2b\mu\phi^2-(a\tau+b\mu)(1+\mu^2)=\phi\mu(1+\mu^2).
\end{equation}

From \eqref{eq-H}, \eqref{eq-bk-phi''mu}, and \eqref{eq-expression2}, we obtain
\begin{equation} \label{eq-Hagain}
H=\frac\rho\phi\frac{(\phi^2k-\mu(1+\mu^2))(h^2+r^2)+\mu(1+\mu^2)\phi^2}{2(\tau^2+h^2(1+\mu^2))}\cdot
\end{equation}

In the next result, we use the above expression of $H$ to ensure the existence of
one-parameter families of complete $h$-pitched helicoidal surfaces with
prescribed mean curvature functions on $\R^2$.

\begin{proposition} \label{th-prescribedH}
For any smooth function $\Psi\colon\R^2\rightarrow\R$ and any
constant $h>0,$ there exists a
one-parameter family of complete
helicoidal surfaces of pitch $h$ in $\h^2\times\R$,
each of them with mean curvature function $H$ satisfying
$$H(X(u,v))=\Psi(\tau(u),\mu(u)),$$
where $X$ is the parameterization given
in~\eqref{eq-parametrization} and $\tau$ and $\mu$
are as in~\eqref{eq-tau&mu}.
\end{proposition}

\begin{proof}
 From equalities \eqref{eq-phiandphi'}--\eqref{eq-abc},
 $\phi$ and $\rho$ are differentiable functions of $\tau$ and $\mu.$
 Hence, considering equality~\eqref{eq-Hagain} for the given function $H=H(\tau,\mu)$ and solving for
 $k$ (notice that the coefficient of $k$ in~\eqref{eq-Hagain} is positive), we have that
 $k=k(\tau,\mu)$ is a smooth function of $(\tau,\mu)\in\R^2.$
 Thus, we can apply~\cite[Lemma 3.2]{halldorsson} to conclude that there exists a
 {one-parameter} family $\mathscr F$
 of plane curves defined on the whole line $\R$, each of them
 with curvature function $k.$

 Therefore, given $h>0$, for any curve $\alpha:\R\rightarrow\R^2$ of $\mathscr F$,
 the $h$-pitched helicoidal surface of $\h^2\times\R$  which is determined by the curve
 $$\sigma=(\alpha,\phi), \,\,\, \phi=\sqrt{1+\|\alpha\|^2},$$
 has mean curvature function $H=H(\tau,\mu)$, as we wished to prove.
\end{proof}

Suppose that $\Sigma=X(\R^2)$ is a helicoidal surface of pitch $h>0$
as in \eqref{eq-parametrization}.
Then, $J\pi(X)=e^{vJ}J\alpha$ and $\eta$ {is given} as in \eqref{eq-normal}, so that
equality \eqref{eq-rotatorH2xR} takes the form:
\begin{equation} \label{eq-Hagain2}
H=\rho(b\tau-a\mu)=\frac{\rho\tau}{\phi}\cdot
\end{equation}

Therefore, since $\rho$ and $\phi$ are functions of $\tau$ and $\mu$, we have

\begin{lemma} \label{lem-conditionrotatorH2xR}
A helicoidal surface $\Sigma=X(\R^2)$ of pitch $h>0$ parameterized as in~\eqref{eq-parametrization}
is a rotator to {\rm MCF} in \,$\h^2\times\R$ if and only if its mean curvature function
$H=H(\tau,\mu)$ satisfies \eqref{eq-Hagain2}.
\end{lemma}

Now, we are in position to provide the

\begin{proof}[Proof of Theorem~\ref{th-main}]
The existence part of the statement follows directly from
Proposition~\ref{th-prescribedH} and Lemma~\ref{lem-conditionrotatorH2xR}.
To prove the asserted properties of the generating curve $\sigma=(\alpha,\phi),$
let us first observe that, by~\eqref{eq-Hagain} and~\eqref{eq-Hagain2},
the curvature of $\alpha$ is the function
\begin{equation} \label{eq-kproof}
k=\frac{2(\tau^2+h^2(1+\mu^2))\tau+(h^2-1)(1+\mu^2)\mu}{(1+r^2)(h^2+r^2)}\cdot
\end{equation}
Then, by combining equalities~\eqref{eq-tau'&mu'} and~\eqref{eq-kproof}, one concludes
that the functions $\tau$ and $\mu$ are solutions of the following ODE system:
\begin{equation} \label{eq-system007}
\left\{
\begin{aligned}
\tau'&=1+\frac{2(\tau^2+h^2(1+\mu^2))\tau\mu+(h^2-1)(1+\mu^2)\mu^2}{(1+r^2)(h^2+r^2)},\\
\mu'&=-\frac{2(\tau^2+h^2(1+\mu^2))\tau^2+(h^2-1)(1+\mu^2)\tau\mu}{(1+r^2)(h^2+r^2)}\cdot
\end{aligned}
\right.
\end{equation}

\begin{figure}[htbp]
\includegraphics[scale=.4]{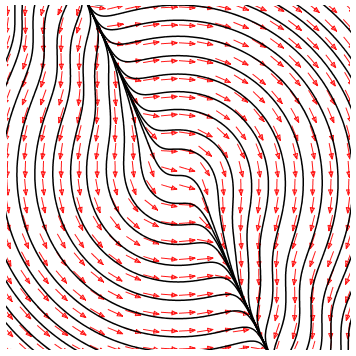}
\caption{\small Phase portrait of system \eqref{eq-system007} for $h=2.$}
\label{fig-phaseportrait3}
\end{figure}

We shall verify the properties of $\sigma$ by establishing the
asymptotic behavior of the solutions
$(\tau,\mu)$ of \eqref{eq-system007} as suggested in
Figure~\ref{fig-phaseportrait3}. This will
be done through the following chain of claims. We note that some of the arguments to prove such claims
are analogous to the ones presented in the proof of Theorem 4.4 of~\cite{dLRS}. Nonetheless,
they will be presented  for completeness and for the reader’s convenience as well.


\begin{claim} \label{claim-noconstantsolutionsH2xR}
The {\rm ODE} system~\eqref{eq-system007} has no constant solutions,
and all solutions are defined on $\R.$
\end{claim}
\begin{proof}[Proof of Claim~\ref{claim-noconstantsolutionsH2xR}]
Let  us assume, by contradiction, that there exists a constant solution
$\psi(s)=(\tau_0,\mu_0), \, s\in\R.$ Since the derivatives of $\tau$ and 
$\mu$ vanish identically,  we have
from~\eqref{eq-tau'&mu'} that $k_0:=k(\tau_0,\mu_0)$ satisfies
$k_0\mu_0=-1$ and $k_0\tau_0=0,$ which yields $\tau_0=0$
and $\mu_0\ne0.$ But, from the first equation in~\eqref{eq-system007},
one has
\begin{equation} \label{eq-tau'}
\tau'=1+\frac{(h^2-1)\mu_0^2}{h^2+\mu_0^2}=\frac{h^2(1+\mu_0^2)}{h^2+\mu_0^2}>0,
\end{equation}
which is a contradiction. Hence, the system~\eqref{eq-system007} has no constant solutions.
From this fact, and since $k=k(\tau,\mu)$ is defined on $\R^2$, we have that any solution
of~\eqref{eq-system007}  is defined on $\R.$
\end{proof}


\begin{claim} \label{claim-reverselimt02}
Suppose that for any solution $\psi(s):=(\tau(s),\mu(s))$ of the system \eqref{eq-system007}, the limit
$\lim_{s\to+\infty}\tau(s)$ (resp. $\lim_{s\to+\infty}\mu(s)$) exists and satisfies:
\[
\lim_{s\to+\infty}\tau(s)=L \,\,\, (\text{resp.}\,\, \lim_{s\to+\infty}\mu(s)=L),
\]
where $L$ is independent of $\psi$ and $-\infty\le L\le+\infty.$ Then,
$\lim_{s\to-\infty}\tau(s)$ (resp. $\lim_{s\to-\infty}\mu(s)$) exists and satisfies:
\[
\lim_{s\to-\infty}\tau(s)=-L \,\,\, (\text{resp.}\,\, \lim_{s\to-\infty}\mu(s)=-L).
\]
\end{claim}
\begin{proof}[Proof of Claim~\ref{claim-reverselimt02}]
Let $\psi(s):=(\tau(s),\mu(s))$ be a solution
of the system~\eqref{eq-system007}. Then,
it is easily checked that $\overbar\psi(s):=-\psi(-s)$ is also
a solution of that system.
Setting $\overbar\psi=(\overbar\tau,\overbar\mu),$
we have that $\overbar\tau(s)=-\tau(-s)$
and $\overbar\mu(s)=-\mu(-s).$ By hypothesis,
$\lim_{s\to +\infty}\overbar\tau(s)$ exists and the first
part of the claim follows from noticing  that
$\lim_{s\to -\infty}\mu(s)
= -\lim_{s\to+\infty} \overbar\mu(s)$.
The remainder of the proof is  analogous and will
be omitted.
\end{proof}


\begin{claim} \label{claim-tauhasonezero02}
The function $\tau$ has precisely one zero $s_0,$
being negative in $(-\infty, s_0)$ and positive in
$(s_0,+\infty).$ As a consequence, the function
$r^2=\tau^2+\mu^2$ has a global minimum and satisfies
$\lim_{s\rightarrow\pm\infty}r^2=+\infty.$
\end{claim}

\begin{proof}[Proof of Claim~\ref{claim-tauhasonezero02}]
Firstly, observe that the equalities~\eqref{eq-tau'&mu'} give
\[
(r^2)'=2(\tau\tau'+\mu\mu')=2(\tau(1+k\mu)+\mu(-k\tau))=2\tau,
\]
which implies that the zeroes of $\tau$ are the critical points
of $r^2$. Also, as seen in the first part of the proof of
Claim~\ref{claim-noconstantsolutionsH2xR},
if $\tau(s_0)=0$ for some $s_0,$ then $\tau'(s_0)>0,$
which implies that $\tau$ has at most one zero $s_0,$
in which case $\tau$ is
negative in $(-\infty, s_0),$ and positive in
$(s_0,+\infty).$

Now, arguing by contradiction, 
we assume that $\tau$ has no zeroes.
We will also assume that $\tau>0$ on $\R,$ since
the complementary case $\tau<0$ can be treated analogously.
Under this assumption, the function
$r^2$ is strictly increasing. So, there exists $\delta\ge 0$ such that
$$
\lim_{s\to -\infty}r^2(s)=\delta.
$$
In particular, since $\tau=\frac{(r^2)'}{2}$, we also have that
\begin{equation}\label{eqtaugoestozero}
\lim_{s\to -\infty}\tau(s)=0,
\end{equation}
which yields $\mu^2\rightarrow\delta$ as $s\rightarrow-\infty.$
However, the first equality in~\eqref{eq-system007} yields
$\lim_{s\rightarrow-\infty}\tau'(s)>0,$
which contradicts~\eqref{eqtaugoestozero}, proving
that $\tau$ has exactly one
zero and that $r^2$ has only one critical point. Consequently,
both the limits of
$r^2$ as $s\rightarrow\pm\infty$ exist in $[0,+\infty]$.

To finish the proof of the claim, we have just to observe  that
if either $\lim_{s\to -\infty}r^2=\delta$ or
$\lim_{s\to +\infty}r^2=\delta$ for some
$\delta>0$, the same arguments as before
lead to a contradiction. Hence,
$\lim_{s\to \pm\infty}r^2(s)=+\infty$.
\end{proof}


\begin{claim} \label{claim-kzeros}
The curvature  $k$ has at most one zero $s_1$. If so,
$k$ is negative in $(-\infty, s_1)$ and positive in
$(s_1,+\infty).$
\end{claim}

\begin{proof}[Proof of Claim~\ref{claim-kzeros}]
Assume that, for some $s_1\in\R,$ we have  $k(s_1)=0$.
Then, by differentiating~\eqref{eq-kproof}, we get
\[
k'(s_1)=\frac{6\tau^2(s_1)+2h^2(1+\mu^2(s_1))}{(1+r^2(s_1))(h^2+r^2(s_1))}>0,
\]
from which the claim clearly follows.
\end{proof}


\begin{claim} \label{claim-limitsexist}
The limits of $\tau$ and $\mu$ as $s\rightarrow\pm\infty$ exist (possibly being infinite).
\end{claim}
\begin{proof}[Proof of Claim~\ref{claim-limitsexist}]
It follows from Claims~\ref{claim-tauhasonezero02} and~\ref{claim-kzeros}
that $\mu'=-k\tau$ has at most two zeroes.
Thus, $\lim_{s\to\pm\infty}\mu$ are both well defined.

Concerning $\tau$,  assume by contradiction that its
limit as $s\rightarrow+\infty$ does not exist.
In this case, for some $\tau_0>0,$ there exists a
strictly increasing
sequence $(s_n)_{n\in\N}$ diverging to $+\infty$
such that (see~Fig.~\ref{fig-taugraph})
\[
\tau(s_n)=\tau_0 \quad\text{and}\quad \tau'(s_n)\tau'(s_{n+1})<0 \quad \forall n\in\N.
\]

\begin{figure}[htbp]
\includegraphics{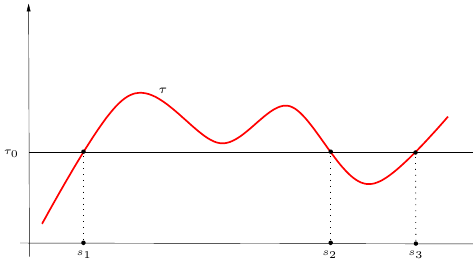}
\caption{\small Graph of $\tau.$}
\label{fig-taugraph}
\end{figure}

Claim~\ref{claim-tauhasonezero02} implies that
$\lim r^2(s_n) = +\infty$, and so
$\lim\mu^2(s_n)= +\infty$.
In this case, our previous
arguments show that either
$\lim\mu(s_n)= +\infty$ or
$\lim\mu(s_n)= -\infty$.
However, we have from~\eqref{eq-kproof} that
(consider the highest powers of $\mu(s_n)$ only)
$$\lim(k(s_n)\mu(s_n))=
\lim\frac{(h^2-1)\mu(s_n)^4}{\mu(s_n)^4}=h^2-1,$$
which, together with~\eqref{eq-tau'&mu'}, yields
$$\lim\tau'(s_n)=\lim(1+k(s_n)\mu(s_n))=h^2>0.$$

It follows from the above inequality that,  for any sufficiently large $n\in\N,$
$\tau'(s_n)$ is positive, which contradicts the fact that $(\tau'(s_n))_{n\in\N}$
is an alternating sequence.
Therefore, $\lim_{s\to+\infty}\tau(s)$ exists.
Since $(\tau,\mu)$ is an arbitrary solution
of~\eqref{eq-system007}, Claim~\ref{claim-reverselimt02}
implies that $\lim_{s\to-\infty}\tau(s)$ also exists,
thereby finishing the proof of the claim.
\end{proof}

\begin{claim} \label{claim-taumulimits}
$\displaystyle\lim_{s\to \pm\infty}\tau(s)=\pm\infty$ \,and\, $\displaystyle\lim_{s\to \pm\infty}\mu(s)=\mp\infty.$
\end{claim}
\begin{proof}[Proof of Claim~\ref{claim-taumulimits}]
By Claim~\ref{claim-limitsexist}, all the limits above exist. 
Regarding the function $\mu$, assume by contradiction that 
$\lim_{s\to +\infty}\mu(s) = L\in \R$.
Under this assumption, we have from
Claims~\ref{claim-tauhasonezero02} and~\ref{claim-limitsexist} that
$\lim_{s\to +\infty} \tau(s)=+\infty$.
Then, it follows from the second equation
in~\eqref{eq-system007} that
$\lim_{s\to +\infty}\mu'(s)= -2\ne 0$, which
contradicts that $L\in\R$.

Suppose now that $\lim_{s\to+\infty}\mu(s)=+\infty$.
Then, there exists  $\bar s\in\R$ such that $k>0$ on $(\bar s,+\infty)$.
Indeed, assuming otherwise, we have from  Claim~\ref{claim-kzeros}
that $k$ must be strictly negative in $(-\infty,+\infty)$. In this case,
considering the unique zero $s_0$ of $\tau$
(cf.~Claim~\ref{claim-tauhasonezero02}), we have that
\[
\mu'(s_0)=-k(s_0)\tau(s_0)=0 \quad\text{and}\quad \mu''(s_0)=-k(s_0)\tau'(s_0)>0,
\]
where, in the last inequality, we used~\eqref{eq-tau'}.
Also,
since $(\tau,\mu)$ is an arbitrary integral curve
of~\eqref{eq-system007}, we have from Claim~\ref{claim-reverselimt02} that
$\lim_{s\to-\infty}\mu(s)=-\infty$, which implies that
$\mu$ must have a local maximum at some point $s_1<s_0$.
Therefore, $$0=\mu'(s_1)=-k(s_1)\tau(s_1),$$ which yields
$k(s_1)=0$, since $s_0$ is the unique zero of $\tau$.
This contradicts our hypothesis on $k$, proving the existence
of $\bar s$ as asserted. However, for any point 
$s\in (s_2,+\infty)$, where $s_2 = \textrm{max} \{ s_0, \bar s \}$,
one has $\mu'(s)=-k(s)\tau(s)<0,$ which contradicts our assumption on $\mu$.
Thus, $\lim_{s\to+\infty}\mu(s)=-\infty$ and,
from Claim~\ref{claim-reverselimt02}, $\lim_{s\to-\infty}\mu(s)=+\infty.$

Finally, suppose that $0\le\lim_{s\to+\infty}\tau(s)=L<+\infty.$
Then, $\lim_{s\to+\infty}\tau'(s)=0.$ But, considering that
$\mu$ has infinite limit as $s\to+\infty$, a computation as in the
final part of the proof of Claim~\ref{claim-limitsexist} gives
that $\lim_{s\to+\infty}\tau'(s)= h^2>0$, which is
a contradiction. This,  together with Claim~\ref{claim-reverselimt02},
shows that $\lim_{s\to\pm\infty}\tau(s)=\pm\infty$.
\end{proof}

\begin{claim} \label{claim-tau/mulimited}
The function $\nu:=-\tau/\mu$ is bounded outside of a compact interval.
\end{claim}

\begin{proof}[Proof of Claim~\ref{claim-tau/mulimited}]
It follows from Claim~\ref{claim-taumulimits} that $\nu$ is well defined and positive at any point
outside of a compact interval of $\R.$ Moreover, at such a point, one has
\begin{equation}\label{eq-derivativemu/tau}
 \nu'=-\frac{\mu+kr^2}{\mu^2}\cdot
\end{equation}

Now, assume by contradiction that there exists a sequence $(s_n)_{n\in\N}$ in $\R$ diverging to infinity, 
such that $\lim\nu(s_n)=+\infty.$ We can also assume, without loss of generality, that
$\nu'(s_n)>0\,\forall n\in\N.$ However, considering~\eqref{eq-kproof}, Claim~\ref{claim-taumulimits},
and the fact that $\lim(-\mu(s_n)/\tau(s_n))=0$, we easily conclude that
\[
\lim(\mu(s_n)+k(s_n)r^2(s_n))=+\infty. 
\]
Thus, for all sufficiently large $n$,
$\nu'(s_n)<0$, which is contradicts  our hypothesis.

Analogously, we derive a contradiction by assuming that there exists
$s_n\rightarrow-\infty$ such that $\nu(s_n)\rightarrow+\infty.$
This proves Claim~\ref{claim-tau/mulimited}.
\end{proof}

In what follows, we shall denote by $\omega=\omega(s)$ the angle function of
$\alpha,$ i.e.,
\[
\alpha=r(\cos\omega,\sin\omega).
\]
It then follows from~\eqref{eq-tau'&mu'} that the equality
\begin{equation} \label{eq-Tandomega'}
T=\frac{\tau}{r^2}\alpha+\omega'J\alpha
\end{equation}
holds at any point where $r\ne 0.$

\begin{claim} \label{claim-infiniteangle}
The angle function $\omega$ of $\alpha$ satisfies
$\lim_{s\to\pm\infty}\omega(s)=+\infty.$
\end{claim}
\begin{proof}[Proof of Claim~\ref{claim-infiniteangle}]
Considering \eqref{eq-Tandomega'} and the equality $(r^2)'=2\tau,$ one has
\[
r'=\frac{\tau}{r} \quad\text{and}\quad \omega'=-\frac{\mu}{r^2}\,\cdot
\]
Hence, given a differentiable function $\varphi=\varphi(r),$ $r\in (0,+\infty),$ its derivative with respect to 
$\omega$ can be written as 
\begin{equation} \label{eq-derivativevarphiH3}
\frac{d\varphi}{d\omega}=\frac{d\varphi}{dr}\frac{dr}{ds}\frac{ds}{d\omega}=-r\varphi'(r)\frac{\tau}{\mu}\cdot
\end{equation}

Next, define $\varphi(r)=\log(\log r).$ Then, $\varphi(r)\rightarrow+\infty$ as $r\rightarrow+\infty$ and
\begin{equation} \label{eq-rvarphi'H3}
r\varphi'(r)=\frac1{\log r}\rightarrow 0 \,\,\, \text{as} \,\,\, r\rightarrow+\infty.
\end{equation}
Since, by Claim~\ref{claim-tau/mulimited}, $-\tau/\mu$ is bounded outside of a compact interval,
it follows from Claim~\ref{claim-tauhasonezero02}
and~\eqref{eq-derivativevarphiH3}--\eqref{eq-rvarphi'H3} that ${d\varphi}/{d\omega}\rightarrow 0$
as $s\rightarrow\pm\infty.$ Thus,
${d\omega}/d\varphi\rightarrow+\infty$ as $s\rightarrow\pm\infty,$
which proves Claim~\ref{claim-infiniteangle}.
\end{proof}

It follows from the above claims that
the trace of $\alpha$ has one point $p_0$ closest to the origin
(Claim \ref{claim-tauhasonezero02}),
and consists of two properly embedded arms centered at
$p_0$ which proceed to infinity by spiraling around the origin
(Claim~\ref{claim-limitsexist}).
In particular, each arm of $\alpha$ gives rise to an embedded arm of the generating curve $\sigma=(\alpha,\phi)$
of $\Sigma,$ which spirals around the $x_3$-axis. In addition,
from Claim \ref{claim-tauhasonezero02}, we have that
$\phi^2=1+r^2\rightarrow+\infty$ as $s\rightarrow\pm\infty,$ which implies that
both arms of $\sigma$ have infinite height, being therefore properly embedded (Fig.~\ref{fig-generatingcurve}).
This concludes our proof. 
\end{proof}

\begin{figure}
\includegraphics[scale=.7]{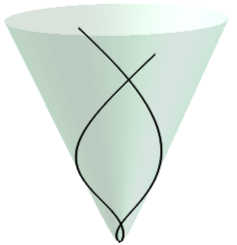}
\caption{\small Generating curve of a $1$-pitched helicoidal rotator-translator of $\h^2\times\R$
in the hyperboloid model of $\h^2$.}
\label{fig-generatingcurve}
\end{figure}

\begin{proof}[Proof of Theorem~\ref{th-translator}]
Consider a helicoidal surface  $\Sigma=X(\R^2)$ of pitch $h>0$ in $\h^2\times\R$
as given in \eqref{eq-parametrization}, and let
$\mathcal G=\{\Gamma_t\,;\, t\in\R\}\subset{\rm Iso}(\h^2\times\R)$ be the group of
downward vertical translations of constant speed $h,$ i.e.,
$\Gamma_t(p)=\exp_p(-th\partial_t),$ where $\exp$ denotes the exponential
map of $\h^2\times\R,$ and $\partial_t$ is the gradient of the height function
of $\h^2\times\R,$ namely $(p,t)\in\h^2\times\R\mapsto t\in\R.$

In the above setting, one has
\[
\frac{\partial\Gamma_t}{\partial t}(p)=d\exp_p(-th\partial_t)(-h\partial_t),
\]
so that the Killing field on $\h^2\times\R$ determined by $\mathcal G$ is
$\xi:=-h\partial_t.$ Then, considering the unit normal to $\Sigma$
as given in \eqref{eq-normal}, we have
\begin{equation*}
 \langle\xi(X),\eta\rangle = -h\rho c=\rho\frac{\tau}\phi,
\end{equation*}
which implies that $\Sigma$ is a $\mathcal G$-soliton if and only if
its mean curvature function is given by
$H=\rho{\tau}/\phi.$ From this and Lemma~\ref{lem-conditionrotatorH2xR},
the result follows.
\end{proof}

\end{document}